\documentclass{amsart}

\usepackage{amsmath,amssymb,amsthm,amsfonts,mathrsfs}

\usepackage[colorlinks=blue,
linkcolor=blue,
anchorcolor=blue,
citecolor=blue
]{hyperref}

\newtheorem{theorem}{Theorem}[section]
\newtheorem*{theorem*}{Theorem}

\newtheorem{corollary}[theorem]{Corollary}
\newtheorem*{corollary*}{Corollary}
\newtheorem{lemma}[theorem]{Lemma}
\newtheorem*{lemma*}{Lemma}
\newtheorem{proposition}[theorem]{Proposition}
\newtheorem*{proposition*}{Proposition}
\theoremstyle{definition}

\theoremstyle{remark}
\newtheorem{remark}[theorem]{Remark}
\newtheorem{example}[theorem]{Example}

\newcommand{\N}{\mathbb{N}}
\newcommand{\Z}{\mathbb{Z}}

\newcommand{\Q}{\mathbb{Q}}

\newcommand{\C}{\mathbb{C}}

\renewcommand{\P}{\mathbb{P}}

\newcommand{\cO}{\mathcal{O}}

\newcommand{\cP}{\mathcal{P}}

\newcommand{\fkp}{\mathfrak{p}}

\renewcommand{\leq}{\leqslant}
\renewcommand{\geq}{\geqslant}

\newcommand{\of}[1]{\left(#1\right)}
\newcommand{\set}[1]{\left\{#1\right\}}
\newcommand{\abs}[1]{\left\vert#1\right\vert}
\newcommand{\norm}[1]{\left\Vert#1\right\Vert}

\newcommand{\smat}[1]{\begin{smallmatrix}#1\end{smallmatrix}}

\newcommand{\on}{\operatorname}
\renewcommand{\Re}{\on{Re}}
\renewcommand{\Im}{\on{Im}}

\newcommand{\Gal}{\on{Gal}}

\begin{document}
\title{Analogues of Alladi's formula}
\author{Biao Wang}
\date{\today}
\maketitle

\begin{abstract}
	In this note, we mainly show the analogue of one of  Alladi's formulas over $\mathbb{Q}$ with respect to the Dirichlet convolutions involving the M\"{o}bius function $\mu(n)$, which is related to the natural densities of sets of primes by recent work of Dawsey, Sweeting and Woo, and Kural et al. This  would give us several new analogues. In particular, we get that if $(k, \ell)=1$, then
	$$-\sum_{\begin{smallmatrix}n\geq 2\\ p(n)\equiv \ell (\operatorname{mod} k)\end{smallmatrix}}\frac{\mu(n)}{\varphi(n)}=\frac1{\varphi(k)},$$
	where  $p(n)$ is the smallest prime divisor of $n$, and $\varphi(n)$ is Euler's totient function. This refines one of Hardy's formulas in 1921. At the end, we give some examples for the $\varphi(n)$ replaced by functions ``near $n$'', which include the sum-of-divisors function.
\end{abstract}

\section{Introduction and statement of results}

It is well known (e.g.,  \cite[(8.8)]{mv07}) that the prime number theorem is equivalent to the assertion that
\begin{equation}\label{pntmu}
	\sum_{n=1}^\infty\frac{\mu(n)}{n}=0,
\end{equation}
where $\mu(n)$ is  the M\"{o}bius function defined by $\mu(n)=(-1)^k$ if $n$ is the product of $k$ distinct primes and zero otherwise.  

In 1977, Alladi \cite{a77} rewrote (\ref{pntmu}) as
\begin{equation}\label{pnt}
	-\sum_{n=2}^\infty\frac{\mu(n)}{n}=1,
\end{equation}
and showed that if $(\ell,k)=1$, then
\begin{equation}\label{alladi}
	-\sum_{\smat{n\geq 2\\ p(n)\equiv \ell (\on{mod}k)}}\frac{\mu(n)}{n}=\frac1{\varphi(k)},
\end{equation}
where $p(n)$ is the smallest prime divisor of $n$, and $\varphi$ is Euler's totient function. 

Alladi's formula (\ref{alladi}) shows a relationship between the M\"{o}bius function $\mu(n)$ and the density of primes in arithmetic progressions.   In 2017, Dawsey \cite{d17} first generalized (\ref{alladi}) to the setting of Chebotarev densities for finite Galois extensions of $\Q$.  Then Sweeting and Woo \cite{sw19}  generalized   Dawsey's result to number fields.  Recently, Kural et al. \cite{kms19}  generalized  all these results to natural densities of sets of primes, see section \ref{secdensity}.  Over $\Q$, the result of Kural et al. is stated as follows.

Let $\cP$ be the set of all primes. Let $S\subseteq \cP$ be a subset of primes. If $S\subseteq \cP$ has a natural density $\delta(S)$, then
\begin{equation}\label{kmsmueq}
	-\sum_{\smat{n\geq 2\\ p(n)\in S}}\frac{\mu(n)}{n}=\delta(S).
\end{equation}
Here we say $S$ has \textit{natural density} $\delta(S)$ \cite{d17,kms19} if the following limit exits:
\begin{equation}
	\delta(S):=\lim_{x\to\infty} \frac{\pi_S(x)}{\pi(x)},
\end{equation}
where $\pi_S(x)=\#\set{p\in S: p\leq x}$, and $\pi(x)=\pi_\cP(x)$.  

For the arithmetic functions other than $\mu$, we \cite{w19,w19rm} showed the analogues of Alladi's and Dawsey's results with respect to the Liouville function and the Ramanujan sum.
In this note, we will unify these two analogues by showing the following analogue of  formula (\ref{kmsmueq}) with respect to the Dirichlet convolutions involving the M\"{o}bius function.  As a corollary, we get a new type of analogues of (\ref{kmsmueq}).

\begin{theorem}\label{mainthm} 
	Suppose $a:\N\to \C$ is an arithmetic function satisfying $a(1)=1$ and $\sum_{n=2}^\infty \frac{|a(n)|}{n}\log\log n<\infty$. Let $\mu*a$ be the Dirichlet convolution of $\mu$ and $a$.
	If $S\subseteq \cP$ has a natural density $\delta(S)$, then
	\begin{equation}\label{mainthmeq}
		-\sum_{\smat{n\geq 2\\ p(n)\in S}}\frac{\mu*a(n)}{n}=\delta(S).
	\end{equation}
\end{theorem}

\begin{corollary}\label{maincor}
	Suppose $a:\N\to \C$ is an arithmetic function satisfying $a(1)=1$ and $|a(n)|\ll{n^{-\alpha}}$ for some $\alpha>0$.
	If $S\subseteq \cP$ has a natural density $\delta(S)$, then
	\begin{equation}\label{maincoreq}
		-\sum_{\smat{n\geq 2\\ p(n)\in S}}\frac{\mu*a(n)}{\varphi(n)}=\delta(S).
	\end{equation}
	
	In particular, for any integer $m\geq1$,  if  $(\ell,k)=1$, then
	\begin{align}
		-\sum_{\smat{n\geq 2\\ p(n)\equiv \ell (\on{mod}k)}}\frac{\mu(n)}{\varphi(n)}&=\frac1{\varphi(k)}, \label{maincoreq1}\\
		-\sum_{\smat{n\geq 2\\ p(n)\equiv \ell (\on{mod}k)}}\frac{c_n(m)}{\varphi(n)}&=\frac1{\varphi(k)}.\label{maincoreq2}
	\end{align}
	Here $c_n(m)=\sum_{\smat{1\leq q\leq n\\(q,n)=1}}e^{\frac{2\pi i qm}n}$ is the Ramanujan sum.
\end{corollary}

\begin{remark} In 1921, Hardy \cite[(8.1), (8.2)]{h1921} showed that  $\sum_{n=1}^\infty\frac{\mu(n)}{\varphi(n)}=0$ and $\sum_{n=1}^\infty\frac{c_n(m)}{\varphi(n)}=0$. Therefore, (\ref{maincoreq1}) and (\ref{maincoreq2}) is a refinement of Hardy's formula. 
\end{remark}

As an application of (\ref{maincoreq}), we get a new  analogue of Dawsey's result \cite{d17}.

\begin{corollary}\label{dawseyphi}
	Let $m\geq 1$ be an integer. Let $K$ be a finite Galois extension of $\Q$ with Galois group $G = \Gal(K/\Q)$. Then for any conjugacy class $C\subseteq G$, we have 
	\begin{equation}\label{dawseyphieq}
		-\sum_{\smat{n\geq 2\\ \left[\frac{K/\Q}{p(n)}\right]=C}}\frac{c_n(m)}{\varphi(n)}=\frac{|C|}{|G|}.
	\end{equation}
	Here the notation under the sum is defined by $$\left[\frac{K/\Q}{p}\right]:=\set{\left[\frac{K/\Q}{\fkp}\right]: \fkp \subseteq\cO_K  \text{ and } \fkp|p}$$ 
	for any unramified prime $p$, where $\cO_K$ denotes the ring of integers in $K$, $\fkp$ denotes a prime ideal in $\cO_K$, and $\left[\frac{K/\Q}{\fkp}\right]$ denotes the Artin symbol for the Frobenius map.
\end{corollary}
\begin{remark}
	Dawsey's result in \cite{d17} is
	\begin{equation}\label{dawseymueq}
		-\sum_{\smat{n\geq 2\\ \left[\frac{K/\Q}{p(n)}\right]=C}}\frac{\mu(n)}{n}=\frac{|C|}{|G|}.
	\end{equation}
\end{remark}

For the proof of Theorem \ref{mainthm}, we follow the approach of proving Theorem 1 in \cite{w19rm}. Then we apply Theorem \ref{mainthm} to show (\ref{maincoreq}). For (\ref{maincoreq1})-(\ref{dawseyphieq}), we will explain how to derive them from (\ref{maincoreq}) in section \ref{secdensity}. At the end, we  give some examples for the $\varphi(n)$ replaced by  functions ``near $n$". We leave the investigation of generalizing of these results to number fields  to interested readers.

\section{Examples of natural densities}\label{secdensity}

In this section, we give two examples of natural densities and explain how they are used to derive the analogues of Alladi's and Dawsey's results from (\ref{mainthmeq}) or (\ref{maincoreq}). We refer interested readers to \cite{kms19} for more interesting examples. To apply Theorem \ref{mainthm} and Corollary \ref{maincor}, notice that
\begin{equation}
	\sum_{n=1}^\infty\frac{\mu*a(n)}{n^s}=\frac{f(s)}{\zeta(s)}
\end{equation}
for $\Re s>1$, 
where  $\zeta(s)=\sum_{n=1}^{\infty}{n^{-s}}$ for $\Re s>1$ is the Riemann zeta function, and $f(s)=\sum_{n=1}^{\infty}a(n){n^{-s}}$ is the Dirichlet series of $a(n)$. Thus, to get (\ref{mainthmeq}) or (\ref{maincoreq}) for an arithmetic function, it suffices to find out its Dirichlet series, write it as ${f(s)}/{\zeta(s)}$, and check the corresponding convergence conditions on the Dirichlet series of $f(s)$.

\begin{example}
	Let $k,\ell\geq1$ be integers with $(k,\ell)=1$.  Let $$S_{k, \ell}=\set{p\in\cP: p\equiv \ell  (\on{mod }k)}.$$
	Then by the prime number theorem in arithmetic progressions, we have
	$$\delta(S_{k, \ell})=\frac1{\varphi(k)}.$$	
	Hence, Alladi's formula (\ref{alladi}) follows by taking $S=S_{k, \ell}$ in (\ref{kmsmueq}).
	
	In Corollary \ref{maincor}, take $a(1)=1, a(n)=0$ for $n\geq 2$. Then $\mu*a=\mu$. By (\ref{maincoreq}), we have
	\begin{equation}\label{phimueq}
		-\sum_{\smat{n\geq 2\\ p(n)\in S}}\frac{\mu(n)}{\varphi(n)}=\delta(S).
	\end{equation}
	Hence,  (\ref{maincoreq1}) follows  by taking $S=S_{k, \ell}$ in (\ref{phimueq}).
	
	For (\ref{maincoreq2}), notice that the Dirichlet series of $c_n(m)$ is
	\begin{equation}\label{ramads}
		\sum_{n=1}^\infty\frac{c_n(m)}{n^s}=\frac{\sigma_{1-s}(m)}{\zeta(s)},
	\end{equation}
	where $\sigma_s(n)=\sum_{d|n}d^s$ is the sum-of-divisors function of exponent $s$. (For the proof of (\ref{ramads}), see \cite[(1.3)]{h1921}.) Then $f(s)=\sum_{n=1}^\infty{a(n)}{n^{-s}}=\sigma_{1-s}(m)$. It follows that $a(n)=0$ for all $n> m$. By (\ref{maincoreq}), 	we have  that
	\begin{equation}\label{ramaphi}
		-\sum_{\smat{n\geq 2\\ p(n)\in S}}\frac{c_n(m)}{\varphi(n)}=\delta(S).
	\end{equation} 
	Hence, (\ref{maincoreq2}) follows immediately by by taking $S=S_{k, \ell}$ in (\ref{ramaphi}).

	Note that (\ref{maincoreq1}) is the case $m=1$ in (\ref{maincoreq2}) due to the fact that $\mu(n)=c_n(1)$ by \cite[Theorem 271]{hw08}.
\end{example}

\begin{example}
	Let $K$ be a finite Galois extension of $\Q$, and let $G = \Gal(K/\Q)$ be the Galois group. 
	For any conjugacy class $C\subset G$, let
	$$S_C=\set{p\in \cP: p \text{ unramified}, \left[\frac{K/\Q}{p}\right]=C}.$$
	
	Then  Chebotarev density theorem \cite{t26} says that $$\pi_{S_C}(x)=\frac{|C|}{|G|}\frac{x}{\log x}+o(\frac{x}{\log x}),$$
	which gives us that $$\delta(S_C)=\frac{|C|}{|G|}.$$
	Hence, Dawsey's result (\ref{dawseymueq}) follows by taking $S=S_C$ in (\ref{kmsmueq}), and (\ref{dawseyphieq}) follows by taking $S=S_C$ in (\ref{ramaphi}).
	
	Note that as \cite{d17}, (\ref{maincoreq2}) also follows by (\ref{dawseyphieq}), if one takes $K=\Q(\zeta_k)$ where $\zeta_k$ is the $k$-th primitive unit root and $C$ the conjugacy class of $\ell$. 
	
	By Theorem \ref{mainthm} and the fact that $\sum_{n=1}^\infty{\lambda(n)}{n^{-s}}=\frac{\zeta(2s)}{\zeta(s)}$, we have
	\begin{align}
		-\sum_{\smat{n\geq 2\\ p(n)\in S}}\frac{\lambda(n)}{n}&=\delta(S),\label{liouville19}\\
		-\sum_{\smat{n\geq 2\\ p(n)\in S}}\frac{c_n(m)}{n}&=\delta(S),\label{rama19}
	\end{align}
	where $\lambda(n)=(-1)^{\Omega(n)}$ is the Liouville function, and $\Omega(n)=\sum_{p^\alpha||n}\alpha$.
	
	Hence, our results in \cite{w19,w19rm} follow by taking $S=S_C$ in (\ref{liouville19}) and (\ref{rama19}), respectively.

\end{example}

In the following three sections, we are devoted to proving Theorem \ref{mainthm} and Corollary \ref{maincor}.

\section{Duality between prime factors}

To prove Theorem \ref{mainthm}, we require an analogue of the following  Alladi's theorem with respect to $\mu*a$, which reveals a duality relationship between $P(n)$ and $p(n)$, where $P(n)$ is the largest prime divisor of $n$.  It is an intermediate result that can convert the density properties of $P(n)$ into the desired sum (\ref{mainthmeq}). We set $P(1)=1$ but $p(1)=\infty$ for convenience.

\begin{theorem}[{\cite[Theorem 6]{a77}}] \label{mainthmmuf}
	For any bounded function $f$ and constant $\delta$, we have
	\begin{equation}
		\sum_{n\leq x}f(P(n))\sim \delta\cdot x
	\end{equation}
	if and only if
	\begin{equation}\label{mainthmfmueq}
		-\sum_{n=2}^\infty \frac{\mu(n)}{n}f(p(n))=\delta.
	\end{equation}
\end{theorem}

\begin{theorem} \label{mainthmf}
	Let $a$ be an arithmetic function as Theorem \ref{mainthm}. Then for any bounded function $f$ and constant $\delta$, we have
	\begin{equation}
		\sum_{n\leq x}f(P(n))\sim \delta\cdot x
	\end{equation}
	if and only if
	\begin{equation}\label{mainthmfeq}
		-\sum_{n=2}^\infty \frac{\mu*a(n)}{n}f(p(n))=\delta.
	\end{equation}
\end{theorem}

To prove Theorem \ref{mainthmf}, it suffices to show that the difference between the partial sums of (\ref{mainthmfmueq}) and (\ref{mainthmfeq}) is of size $o(1)$.  For this purpose, we need to estimate the following sum
$$R(x,y)=\sum_{\smat{1\leq n\leq x\\ p(n)> y}}\frac{\mu(n)}n, x,y\geq1.$$

\begin{lemma} \label{muupbd1}
	For any $x,y\geq1$, we have that
	\begin{equation}
		|R(x,y)|\leq 1.
	\end{equation}
\end{lemma}
\begin{proof}
	This follows immediately by \cite[Theorem 1]{t10}.
\end{proof}

\begin{lemma}\label{muupbd2}
	For $ 1\leq y\leq e^{\sqrt{\log x}}$, we have that
	\begin{equation}\label{muupbd2eq}
		R(x,y)\ll \exp\big(-c(\log x)^{\frac12}\big).
	\end{equation}
	Here and thereafter mentioned $c$ is  a positive constant that may vary from one line to the next. The implied constant in (\ref{muupbd2eq}) is absolute.
\end{lemma}

\begin{proof}
	If $y=1$, then $R(x,y)=\sum_{n\leq x}\frac{\mu(n)}{n}$. The estimate $\sum_{n\leq x}\frac{\mu(n)}{n}\ll  x\exp(-c\sqrt{\log x})$ is well known(e.g., \cite[(6.18)]{mv07}). 
	
	Suppose $ 2\leq y\leq e^{\sqrt{\log x}}$. Then by \cite[(3.5)]{a82}, we have 
	\begin{equation}\label{mxy0}
		\sum_{\smat{1\leq n\leq x\\ p(n)> y}}\mu(n)\ll x\log x\exp\of{-\frac\alpha2\log\of{\frac\alpha3}}+x\exp\of{-c\sqrt{\log x}},
	\end{equation}
	where $\alpha=\frac{\log x}{\log y}$. Note that $\alpha\geq \sqrt{\log x}$
	for $ 2\leq y\leq e^{\sqrt{\log x}}$. For the first term on the right hand side of (\ref{mxy0}), we have $$x\log x\exp\of{-\frac\alpha2\log\of{\frac\alpha3}}\ll x\log x\exp\of{-2\sqrt{\log x}}\ll x\exp\of{-\sqrt{\log x}}.$$
	So (\ref{mxy0}) gives us that
	\begin{equation}\label{mxy1}
		\sum_{\smat{1\leq n\leq x\\ p(n)> y}}\mu(n)\ll x\exp\of{-c\sqrt{\log x}}.
	\end{equation}
	
	On the other hand, by \cite[Theorem 3]{t10}, we have 
	\begin{equation}\label{ypnt}
		\sum_{\smat{n\geq1\\ p(n)>y}}\frac{\mu(n)}n=0.	
	\end{equation}
	
	Finally, we show (\ref{muupbd2eq}) by using (\ref{mxy1}), (\ref{ypnt}), and summation by parts. Let $y$ be fixed. Put $M_y(x)=\sum_{\smat{1\leq n\leq x\\ p(n)> y}}\mu(n)$. First, from (\ref{ypnt}) we have $R(x,y)=-\sum_{\smat{n> x\\ p(n)>y}}\frac{\mu(n)}n$. Then using summation by parts and (\ref{mxy1}), we get that
	\begin{align*}
		R(x,y)&=-\int_x^\infty \frac{dM_y(t)}{t}=-\left.\frac{M_y(t)}{t}\right|_x^\infty-\int_x^\infty \frac{M_y(t)}{t^2}dt\\
		&=\frac{M_y(x)}{x}-\int_x^\infty \frac{M_y(t)}{t^2}dt\\
		&\ll \exp\of{-c\sqrt{\log x}}+\int_x^\infty\frac{\exp\of{-c\sqrt{\log t}}}{t}dt\\
		&\ll \exp\of{-c\sqrt{\log x}}+ \exp\of{-\frac{c}2\sqrt{\log x}}\cdot\int_x^\infty\frac{\exp\of{-\frac{c}2\sqrt{\log t}}}{t}dt\\
		&\ll \exp\of{-\frac{c}2\sqrt{\log x}},
	\end{align*}
	where the integral $\int_x^\infty\frac{\exp\of{-\frac{c}2\sqrt{\log t}}}{t}dt<\int_2^\infty\frac{\exp\of{-\frac{c}2\sqrt{\log t}}}{t}dt<\infty$ is bounded by a constant. Thus, we get the desired estimate (\ref{muupbd2eq}).
\end{proof}

\paragraph{Proof of Theorem \ref{mainthmf}}
Set   $f(\infty)=0$ for convenience. First, we break up the partial sum of (\ref{mainthmfeq}) into three sums:
\begin{align}
	\sum_{n\leq x}\frac{\mu*a(n)}{n}f(p(n))&=\sum_{n\leq x}\frac{f(p(n))}{n}\sum_{d|n}\mu(d)a\big(\frac{n}{d}\big)=\sum_{d\leq x}\frac{a(d)}{d}\sum_{n\leq \frac{x}{d}}\frac{\mu(n)}{n}f(p(dn))\nonumber\\
	&=\sum_{ n\leq x}\frac{\mu(n)}{n}f(p(n))+\sum_{2\leq d\leq x^{\frac1\alpha}}\frac{a(d)}{d}\sum_{n\leq \frac{x}{d}}\frac{\mu(n)}{n}f(p(dn))\nonumber\\
	&\qquad+\sum_{x^{\frac1\alpha}<d\leq x}\frac{a(d)}{d}\sum_{n\leq \frac{x}{d}}\frac{\mu(n)}{n}f(p(dn))\nonumber\\
	&=S_1+S_2+S_3, \label{mainpfeq}
\end{align}
where $\alpha$ is to be determined later.

Next, we show that $S_2$ and $S_3$ are error terms of size $o(1)$. For the inside sum  in $S_2$ and $S_3$, we  write it as a sum of  $R(x,y)$'s:
\begin{align} 
	\sum_{ n\leq \frac{x}{d}}\frac{\mu(n)}{n}f(p(dn))&=\sum_{\smat{ n\leq \frac{x}{d}\\ p(n)\geq p(d)}}\frac{\mu(n)}{n}f(p(d))+\sum_{\smat{n\leq \frac{x}{d}\\ p(n)< p(d)}}\frac{\mu(n)}{n}f(p(n))\nonumber\\
	&=f(p(d))\sum_{\smat{n\leq \frac{x}{d}\\ p(n)\geq p(d)}}\frac{\mu(n)}{n}+\sum_{p<p(d)}f(p)\sum_{\smat{ n\leq \frac{x}{d}\\ p(n)=p}}\frac{\mu(n)}{n}\nonumber\\
	&=f(p(d))\sum_{\smat{ n\leq \frac{x}{d}\\ p(n)\geq p(d)}}\frac{\mu(n)}{n}-\sum_{p<p(d)}\frac{f(p)}{p}\sum_{\smat{n\leq \frac{x}{pd}\\ p(n)>p}}\frac{\mu(n)}{n}\nonumber\\
	&=f(p(d)){R\Big(\frac{x}{d}, p(d)-1\Big)}-\sum_{p<p(d)}\frac{f(p)}{p}{R\Big(\frac{x}{pd},p\Big)}. \label{pfeq2}
\end{align}

Since $f$ is bounded, from (\ref{pfeq2}) we obtain that
\begin{equation}\label{pfinside}
	\sum_{ n\leq \frac{x}{d}}\frac{\mu(n)}{n}f(p(dn))\ll \abs{R\of{\frac{x}{d}, p(d)-1}}+\sum_{p<p(d)}p^{-1} \abs{R\of{\frac{x}{pd},p}}.
\end{equation}
We will use the estimates of $R(x,y)$ to estimate (\ref{pfinside}). 

To apply Lemma \ref{muupbd2} and require $x^{\frac1\alpha}\to\infty$ as $x\to\infty$, we take $\alpha=(\log x)^{\frac34}$. 

For $S_2$, we have $2\leq d\leq x^{\frac1\alpha}$. Suppose $x$ is sufficiently large. Then for $p<d\leq x^{\frac1\alpha}$, we have $\frac{x}{d}\geq\frac{x}{pd}\geq \frac{x}{d^2} \geq x^{1-\frac2\alpha}\geq  x^{\frac12}$. Then $ e^{\sqrt{\log\frac{x}{pd}}}\geq  e^{\sqrt{\frac12\log x}}\geq e^{(\log x)^{\frac14}}=x^{\frac1\alpha}>p$. So we have ${x}/{d}\geq x^{\frac12}$  and $p\leq e^{\sqrt{\log\frac{x}{pd}}}$ for all $p<d\leq x^{\frac1\alpha}$.  By Lemma \ref{muupbd2}, we get that
\begin{align}
	&\quad \abs{R\of{\frac{x}{d}, p(d)-1}}+\sum_{p<p(d)}p^{-1} \abs{R\of{\frac{x}{pd},p}}\nonumber\\
	&\ll \exp\of{-c(\log \frac{x}{d})^{\frac12}}
	+\sum_{p<p(d)}p^{-1}\exp\Big(-c(\log \frac{x}{pd})^{\frac12}\Big)\nonumber\\
	&\ll \exp\of{-c(\log x^{\frac12})^{\frac12}}+\sum_{p<p(d)}p^{-1}\exp\Big(-c(\log \frac{x^{\frac12}}{p})^{\frac12}\Big)\nonumber\\	
	&\ll \exp\of{-c(\log x)^{\frac12}}+\sum_{n<x^{\frac14}}\frac1n\exp\Big(-c(\log \frac{x^{\frac12}}{n})^{\frac12}\Big)\nonumber\\
	&\ll \exp\of{-c(\log x)^{\frac12}}. \label{s2-1}
\end{align}
Here for the sum in the third line of (\ref{s2-1}), we extended the index set $\set{p: p<p(d)}$ of primes to the set $\{n: n<x^{\frac14}\}$ of integers by the fact that $p(d)\leq d\leq x^{\frac1\alpha}\leq x^{\frac14}$. And the last estimate for the sum over $n<x^{\frac14}$  follows by \cite[(2.36)]{a77}.

Hence, it follows by (\ref{pfinside})  and (\ref{s2-1}) that
\begin{align}
	S_2&\ll \sum_{2\leq d\leq x^{\frac1\alpha}}\frac{|a(d)|}{d} \exp\of{-c(\log x)^{\frac12}}\nonumber\\
	&\ll \Big(\sum_{d=2}^\infty\frac{|a(d)|}{d}\Big) \exp\of{-c(\log x)^{\frac12}}\nonumber\\
	&\ll \exp\of{-c(\log x)^{\frac12}},  \label{mainpfeqs2}
\end{align}
where $\sum_{d=2}^\infty\frac{|a(d)|}{d}<\infty$ is a positive constant due to the assumptions on $a$.

For $S_3$, by Lemma \ref{muupbd1}, we have 
\begin{equation}\label{pfs3-2}
	\abs{R\of{\frac{x}{d}, p(d)-1}}+\sum_{p<p(d)}p^{-1} \abs{R\of{\frac{x}{pd},p}} \leq 1+\sum_{p<p(d)}\frac{1}{p}\ll\log\log d.
\end{equation}
Then plugging (\ref{pfinside}) and  (\ref{pfs3-2})  into $S_3$, we get that 
\begin{equation}\label{mainpfeqs3}
	S_3=O\Big(\sum_{x^{\frac1\alpha}<d\leq x}\frac{|a(d)|}{d}\log\log d\Big)=o(1)
\end{equation}
due to the convergence of $\sum_{d=2}^\infty\frac{|a(d)|}{d}\log\log d$.

Thus, combining (\ref{mainpfeq}),  (\ref{mainpfeqs2})  and  (\ref{mainpfeqs3}) together, we conclude that
\begin{equation}\label{maineq}
	\sum_{n\leq x}\frac{\mu*a(n)}{n}f(p(n))=	\sum_{n\leq x}\frac{\mu(n)}{n}f(p(n))+o(1).
\end{equation}
And Theorem \ref{mainthmf} follows immediately by Theorem \ref{mainthmmuf} and (\ref{maineq}) above. \qed

\section{Proof of Theorem \ref{mainthm}}
Now, we use the following theorem on the  density of the largest prime divisors of integers to derive the desired formula (\ref{mainthmeq}) via Theorem \ref{mainthmf}. 

\begin{theorem}[{\cite[Theorem 3.1]{kms19}}]\label{Pndensity}
	If $S\subseteq \cP$ has a natural density $\delta(S)$, then
	\begin{equation}
		\sum_{\smat{n\leq x\\P(n)\in S}} 1\sim\delta(S)\cdot x.
	\end{equation}
\end{theorem}
\begin{remark}
	Theorem \ref{Pndensity} is the statement of \cite[Theorem 3.1]{kms19} for $K=\Q$.
\end{remark}

\paragraph{Proof of Theorem \ref{mainthm}.}
Let $f(n)$ be the characteristic function  of $S$ defined by
$$f(n)=\begin{cases}
1,& \text{if } n\in S;\\
0,& \text{if } n\notin S.
\end{cases}
$$
Then we can rewrite Theorem \ref{Pndensity} as 
$$\sum_{n\leq x}f(P(n))\sim \delta(S)\cdot x.$$

By Theorem \ref{mainthmf}, we have
$$-\sum_{n=2}^\infty \frac{\mu*a(n)f(p(n))}{n}= \delta(S),$$
which is  exactly (\ref{mainthmeq}).  This completes the proof of Theorem \ref{mainthm}.
\qed

\section{Proof of Corollary \ref{maincor}}

In section \ref{secdensity}, we derived (\ref{maincoreq1}) and  (\ref{maincoreq2}) from (\ref{maincoreq}). In this section, we apply Theorem \ref{mainthm} to show (\ref{maincoreq}). 

Let $b(n)=\sum_{d|n}\mu*a(d)\frac{d}{\varphi(d)}$. Then by the M\"{o}bius inversion formula, we have
\begin{equation}
	\frac{\mu*a(n)}{\varphi(n)}=\frac{\mu*b(n)}{n}.
\end{equation}
Clearly, $b(1)=1$. By Theorem \ref{mainthm}, to prove (\ref{maincoreq}), it suffices to show that
\begin{equation}\label{bnbdeq}
	\sum_{n=1}^\infty \frac{|b(n)|}{n}\log\log n<\infty.
\end{equation}

First, by $\mu*a(d)=\sum_{d_1d_2=d}\mu(d_1)a(d_2)$ we have $$b(n)=\sum_{d|n}\sum_{d_1d_2=d}\mu(d_1)a(d_2)\frac{d}{\varphi(d)}=\sum_{d_1d_2|n}\mu(d_1)a(d_2)\frac{d_1d_2}{\varphi(d_1d_2)}.$$ By $\varphi(de)=\varphi(d)\varphi(e)\frac{(d,e)}{\varphi((d,e))}$, we can rewrite $b(n)$ as
\begin{align}
	b(n)&=\sum_{de|n}\frac{de}{\varphi(de)}\mu(d)a(e)=\sum_{e|n}ea(e)\sum_{d|\frac{n}{e}}\frac{d\mu(d)}{\varphi(de)}\nonumber\\
	&=\sum_{e|n}\frac{e}{\varphi(e)}a(e)\sum_{d|\frac{n}{e}}\frac{d\varphi((d,e))}{\varphi(d) (d,e)}\mu(d). \label{phieq1}
\end{align}

Put $m={n}/{e}$. Notice that $\frac{d\varphi((d,e))}{\varphi(d) (d,e)}=\prod\limits_{p|d, p\nmid e}(1-p^{-1})^{-1}$ for $\varphi(n)=n\prod\limits_{p|n}(1-p^{-1})$. So we can simplify the sum over $d$ as follows:
\begin{align}
	\sum_{d|m}\frac{d\varphi((d,e))}{\varphi(d) (d,e)}\mu(d)&=\sum_{d|m}\mu(d)\prod_{p|d, p\nmid e}\big(1-p^{-1}\big)^{-1}\nonumber\\
	&=\sum_{\smat{d_1|(e,m), d_2|m, d_2\nmid e\\(d_1,d_2)=1}}\mu(d_1)\mu(d_2)\prod_{p|d_2}\big(1-p^{-1}\big)^{-1}\nonumber\\
	&=\sum_{d_1|(e,m)}\mu(d_1)\sum_{\smat{d_2|m, d_2\nmid e\\(d_1,d_2)=1}}\mu(d_2)\prod_{p|d_2}\big(1-p^{-1})^{-1}\nonumber\\
	&=\sum_{d_1|(e,m)}\mu(d_1)\prod_{p|m,p\nmid e}\big(1-\big(1-p^{-1}\big)^{-1}\big)\nonumber\\
	&=\delta_{1,(e,m)}\prod_{p|m, p\nmid e}\big(1-\big(1-p^{-1}\big)^{-1}\big), \label{phieq2}
\end{align}
where $\delta_{1,n}$ is the Kronecker delta  defined by $\delta_{1,1}=1$ if $n=1$ and zero otherwise.

Plugging (\ref{phieq2}) into (\ref{phieq1}), we have
\begin{equation}
	b(n)=\sum_{e|n, (e,\frac{n}{e})=1}\frac{e}{\varphi(e)}a(e)\prod_{p|\frac{n}{e}}\frac{1}{1-p}.
\end{equation}
It follows by the triangle inequality that
\begin{equation}\label{bnbd-1}
	|b(n)|\leq \sum_{e|n, (e,\frac{n}{e})=1}|a(e)|\frac{e}{\varphi(e)}\prod_{p|\frac{n}{e}}{\frac{1}{p-1}}.
\end{equation}

Notice that by the assumptions, we have $|a(n)|\ll{n^{-\alpha}}, \alpha>0$,  and by \cite[Theorem 327]{hw08}, we have $\frac{n}{\varphi(n)}\ll n^{\frac{\alpha}{2}}$. So from (\ref{bnbd-1}), we get the following estimate for $b_n$:
\begin{equation}\label{bnupbd}
	|b(n)|\ll \sum_{d|n}{d^{-\frac{\alpha}{2}}} \prod_{p|\frac{n}{d}}{\frac{1}{p-1}}.
\end{equation}

Put $c(n)=\sum_{d|n}{d^{-\frac{\alpha}{2}}} \prod_{p|\frac{n}{d}}{\frac{1}{p-1}}$. Then $c(n)$ is the Dirichlet convolution of $c_1(n)=n^{-\frac\alpha2}$ and $c_2(n)=\prod_{p|n}\frac1{p-1}$. Clearly, the Dirichlet series of $c_1(n)$ is $\zeta(s+\alpha/2)$, which is absolutely convergent on $\Re s>1-\alpha/2$. As regards the multiplicative function $c_2(n)$,  we have
$$\sum_{p}\sum_{\nu\geq1}\abs{\frac{c_2(p^\nu)}{p^{\nu s}}}=\sum_{p}\frac1{(p-1)(p^\sigma-1)}<\infty$$
for $\sigma=\Re s>0$.
By \cite[Theorem 1.3, \S II.1]{t15}, it follows that the Dirichlet series of $c_2(n)$ is absolutely convergent on $\Re s>0$ and $\sum_{n=1}^\infty \frac{c_2(n)}{n^s}=\prod_{p}(1+\frac1{(p-1)(p^s-1)})$.
Thus, by \cite[Theorem 1.8]{mv07}, the Dirichlet series of $c(n)$ is absolutely convergent on $\Re s>\sigma_0$, where $\sigma_0=\max\set{1-\alpha/2, 0}<1$, and 
\begin{equation}
	\sum_{n=1}^\infty \frac{c(n)}{n^s}=\zeta(s+\frac\alpha2)\prod_{p}\Big(1+\frac1{(p-1)(p^s-1)}\Big).
\end{equation}
Therefore, by \cite[(1.8)]{mv07}, the derivative of $\sum_{n=1}^\infty {c(n)}{n^{-s}}$ is convergent at $s=1$, which implies that $\sum_{n=1}^\infty \frac{c(n)}{n}\log n<\infty$. It follows immediately by (\ref{bnupbd})  that
\begin{equation}\label{bnconv}
	\sum_{n=1}^\infty \frac{|b(n)|}{n}\log n<\infty.
\end{equation}
Hence, by (\ref{bnconv}) we finally obtain (\ref{bnbdeq})  and the conclusion (\ref{maincoreq}). This completes the proof of Corollary \ref{maincor}.
\qed

\section{Examples}
At the end of this note, we give some examples that are analogous to (\ref{maincoreq}) or (\ref{phimueq}) for the $\varphi(n)$ by  replaced by functions ``near $n$''.  We may call them ``mock $n$-functions''.

\begin{example}
	Similar to the proof above, under the assumptions of Corollary \ref{maincor}, we have 
	\begin{equation}\label{psialldi}
		-\sum_{\smat{n\geq 2\\ p(n)\in S}}\frac{\mu*a(n)}{\psi(n)}=\delta(S)
	\end{equation}
	for the Dedekind psi function $\psi(n)=n\prod_{p|n}(1+p^{-1})$. 
\end{example}	

In the following examples, we focus on the M\"{o}bius function $\mu(n)$, leaving the investigation of the case $\mu*a(n)$ to interested readers.

\begin{example}
	Let $\sigma(n)=\sum_{d|n}d$ be the sum-of-divisors  function. Then $\sigma(n)=\psi(n)$, if  $n$ is square-free. Thus, by (\ref{psialldi}) we have 
	\begin{equation}\label{sigmaalladi}
		-\sum_{\smat{n\geq 2\\ p(n)\in S}}\frac{\mu(n)}{\sigma(n)}=\delta(S).
	\end{equation}
	
	In general, if $g(n)$ is a multiplicative function satisfying ${g(p)}/p-1\ll p^{-t}$ for some $t>0$, then $-\sum_{\smat{n\geq 2\\ p(n)\in S}}\frac{\mu(n)}{g(n)}=\delta(S)$ by Theorem \ref{mainthm}. In particular, (\ref{sigmaalladi}) also holds for the  $\sigma(n)$ replaced by the multiplicative function $\frac{\sigma_k(n)}{n^{k-1}}$ for any integer $k\geq 1$.
\end{example}

\begin{example}
	For integer $k\geq 1$, let $r_k(n)=\#\{(m_1,\dots,m_k)\in\Z^k| n=m_1^2+\cdots+m_k^2\}$ be the number of representations of $n$ as the sum of $k$ squares.  Note that $\frac{r_{k}(n)}{2k}$ is multiplicative if and only if $k=1,2,4,8$, see \cite[Theorem 10.3.4]{g85}. By \cite[(9.19)]{g85}, we have $r_4(n)=8\sum_{d|n,4\nmid d}d$ and $r_8(n)=16\sum_{d|n}(-1)^{n-d}d^3$.  Thus,  (\ref{sigmaalladi}) also holds for the  $\sigma(n)$ replaced by $\frac{r_4(n)}8$ or $\frac{r_8(n)}{16n^2}$. 
	
\end{example}	

\begin{example}
	Notice that $\sigma_{2k-1}(n)$ lies in the Fourier expansion of the normalized Eisenstein series
	$$E_{2k}(\tau)=1-\frac{4k}{B_{2k}}\sum_{n=1}^\infty \sigma_{2k-1}(n)q^n, $$
	which is closely related to the theta functions of a lattice. 
	Here $q=e^{2\pi i\tau}$, and $B_{2k}$ are the Bernoulli numbers. 
	The theta function $\Theta_\Gamma(\tau)$ associated to a lattice $\Gamma$ is defined by
	\begin{equation}
		\Theta_\Gamma(\tau)=\sum_{x\in \Gamma}e^{i\pi \tau\norm{x}^2}, \Im \tau>0. 
	\end{equation}
	
	Suppose the Fourier expansion of $	\Theta_\Gamma(\tau)$ is 
	\begin{equation}
		\Theta_\Gamma(\tau)=1+\sum_{n=1}^\infty r_{\Gamma}(n)q^n.
	\end{equation}
	Then by \cite[\S VII6.6]{s73}, for the $E_8$ lattice $\Gamma_8$, we have  $\Theta_{\Gamma_8}(\tau)=E_4(\tau)$; and for the lattice $\Gamma=\Gamma_8\oplus\Gamma_8$ or $\Gamma_{16}$, we have $\Theta_{\Gamma}(\tau)=E_8(\tau)$. Thus, we conclude the following statement which connects the $E_8$ lattice and densities of sets of primes.
	\begin{proposition} Let $k=\dim\Gamma/4$. For $\Gamma=\Gamma_8$, $\Gamma_8\oplus\Gamma_8$, or $\Gamma_{16}$, if $S$ has a natural density, then 
		\begin{equation}\label{theta}
			\frac{4k}{B_{2k}}\sum_{\smat{n\geq 2\\ p(n)\in S}}\frac{\mu(n)n^{2k-2}}{r_\Gamma(n)}=\delta(S).
		\end{equation}
	\end{proposition}
\end{example}

\begin{remark}
	It is unclear if (\ref{theta}) holds for other lattices, say the Leech lattice $\Lambda_{24}$.
\end{remark}

\section{Acknowledgements}
The author would like to thank his advisor Professor Xiaoqing Li for her constant support, and Professor Hui Xue, Ze Xu, and Bingrong Huang for their comments, and Michael Kural, Vaughan McDonald, and Ashwin Sah for bringing their article \cite{kms19} to his attention. The author would also like to thank the anonymous referee  for a careful reading of the paper and helpful corrections and suggestions.

	\

	Department of Mathematics, University at Buffalo, Buffalo, NY 14260, USA
	
	Email: bwang32@buffalo.edu

\end{document}